\documentclass[journal]{IEEEtran}

\ifCLASSOPTIONcompsoc
  \usepackage[nocompress]{cite}
\else
  \usepackage{cite}
\fi

\usepackage{amsmath}
\usepackage{amssymb}
\newcommand{\papertitle}{An entropy-based bound for the computational complexity of a s\witch{} system}

\usepackage{hyperref}

\usepackage[english]{babel}
\usepackage[T1]{fontenc}
\usepackage{todonotes}
\usepackage{xifthen}
\usepackage{todonotes}

\usepackage{graphicx}
\usepackage{epstopdf}

\usepackage{longtable}
\usepackage{tabu}

\usepackage{amsmath}
\usepackage{amssymb}
\usepackage{caption}
\usepackage{subcaption}
\usepackage{multirow}

\usepackage{algorithm}
\usepackage{algorithmic}

\usepackage{tkz-graph}
\usepackage{tikz}

\usepackage{url}

\usetikzlibrary{arrows,matrix,decorations.pathreplacing,positioning,chains,fit,shapes,calc,3d}

\DeclareMathOperator{\newdiff}{d} 
\newcommand{\dif}{\newdiff\!}

\newcommand{\vertiii}[1]{{\left\vert\kern-0.25ex\left\vert\kern-0.25ex\left\vert #1
    \right\vert\kern-0.25ex\right\vert\kern-0.25ex\right\vert}}

\newcommand{\Sect}{Section~}
\newcommand{\eqdef}{\triangleq}

\newcommand{\ktup}[1]{$#1$-tuple}
\newcommand{\ktups}[1]{$#1$-tuples}

\usepackage{mathtools}

\newcommand{\opt}[1]{{#1}^\star}

\newcommand{\R}{\mathbb{R}}

\newcommand{\Rh}{\R_{\mathbf{2d}}}

\newcommand{\pmu}{\widetilde{\mu}}

\newcommand{\Tr}{\top}

\newcommand{\Nodes}{V}

\newcommand{\nodes}{nodes}

\newcommand{\Arcs}{E}
\newcommand{\Arcsin}[1][\Arcs]{{#1}^-}

\newcommand{\arc}{edge}
\newcommand{\arcs}{edges}

\newcommand{\secref}[1]{Section~\ref{sec:#1}}

\newcommand{\figref}[1]{Figure~\ref{fig:#1}}

\newcommand{\gamub}{\overline{\gamma}}
\newcommand{\gamubopt}{\opt{\gamub}}

\newcommand{\A}{\mathcal{A}}

\newcommand{\G}{G}

\newcommand{\jsra}{\rho}
\newcommand{\jsr}{\jsra(\A)}

\newcommand{\ujsrsos}[1][2d]{\jsra_{\text{SOS-}#1}(\A)}
\newcommand{\jsrsossub}[3]{\jsra_{\text{SOS-}#1}(#2,#3)}
\newcommand{\jsrsos}[1][2d]{\jsrsossub{#1}{\G}{\A}}
\newcommand{\cpr}[1][p]{\jsra_{#1}(\G,\A)}

\newcommand{\cdr}{\cpr[2d]}
\newcommand{\cjsr}[1][\G]{\jsra(#1,\A)}

\newcommand{\cjsrk}[1][k]{\hat{\jsra}_{#1}(\G,\A,\|\cdot\|)}
\newcommand{\cjsrp}{\jsra(\G',\A')}
\newcommand{\cjsrkp}[1][k]{\hat{\jsra}_{#1}(\G',\A',\|\cdot\|)}

\newcommand{\witch}{witched}

\newcommand{\iterc}{\tau}

\newcommand{\defAG   }{a finite set of matrices $\A$ constrained by an automaton $\G$}

\newcommand{\defAGn  }{a finite set of matrices $\A \subset \mathbb{R}^{n \times n}$ constrained by an automaton $\G$}
\newcommand{\defAGe  }{a finite set of matrices $\A$ constrained by an automaton $\G(\Nodes,\Arcs)$}

\newcommand{\PrimalFig}[3][0.18]{
  $2d = #2$
  & \NodeFig{#3}{1}{#1}
  & \NodeFig{#3}{2}{#1}
  & \NodeFig{#3}{3}{#1}
  & \NodeFig{#3}{4}{#1}\\
}

\newcommand{\NodeFig}[3]{
  \includegraphics[trim=4.4cm 0 3.1cm 0, clip, width=#3\textwidth]{PEDJ_d#1_v#2.png}
}

\usepackage{amsthm}
\theoremstyle{definition}
\newtheorem{myprop}{Proposition}
\newcommand{\propref}[1]{Proposition~\ref{prop:#1}}
\newtheorem{mytheo}{Theorem}
\newcommand{\theoref}[1]{Theorem~\ref{theo:#1}}
\newtheorem{mylem}{Lemma}
\newcommand{\lemref}[1]{Lemma~\ref{lem:#1}}

\newtheorem{mycoro}{Corollary}
\newcommand{\cororef}[1]{Corollary~\ref{coro:#1}}
\newtheorem{mydef}{Definition}

\newtheorem{myexem}{Example}
\newcommand{\exemref}[1]{Example~\ref{exem:#1}}
\newtheorem{myrem}{Remark}

\newtheorem{myprog}{Program}
\newcommand{\progref}[1]{Program~\ref{prog:#1}}

\usepackage{doi}

\newcommand{\Paths}{\G_k}

\newcommand{\expfrac}[2]{#1/#2}

\begin{document}


\title{\papertitle}

\author{Beno\^it~Legat,
        Pablo~A.~Parrilo,
        and~Rapha\"el~M.~Jungers,
\IEEEcompsocitemizethanks{\IEEEcompsocthanksitem
\IEEEcompsocthanksitem B. Legat and R. M. Jungers are with the ICTEAM, Universit\'e catholique de Louvain
(e-mail: \texttt{benoit.legat@uclouvain.be}; \texttt{raphael.jungers@uclouvain.be}).
\IEEEcompsocthanksitem P. A. Parrilo is with the Laboratory for Information and Decision Systems, Massachusetts Institute of Technology
(e-mail: \texttt{parrilo@mit.edu}).
}
}

\IEEEtitleabstractindextext{%
\begin{abstract}
    The joint spectral radius (JSR) of a set of matrices characterizes the
    maximal asymptotic growth rate of an infinite product of matrices of
    the set. This quantity appears in a number of applications including
    the stability of s\witch{} and hybrid systems. A popular method used for
    the stability analysis of these systems searches for a Lyapunov function with convex optimization tools.

    We analyse the accuracy of this method for \emph{constrained s\witch{} systems},
    a class of systems that has attracted increasing attention recently.
    We provide a new guarantee for the upper bound provided by the sum of squares implementation of the method.
    This guarantee relies on the $p$-radius of the system and the entropy of the language of allowed switching sequences.

    We end this paper with a method to reduce the computation of the JSR of
    low rank matrices to the computation of the constrained JSR of matrices
    of small dimension.
\end{abstract}

}

\maketitle

\IEEEdisplaynontitleabstractindextext

%
\IEEEpeerreviewmaketitle

\ifCLASSOPTIONcompsoc
\IEEEraisesectionheading{\section{Introduction}\label{sec:introduction}}
\else
\section{Introduction}
\label{sec:introduction}
\fi

\label{sec:intro}

%
%
%
%

\IEEEPARstart{I}{n} recent years, the study of the stability of hybrid systems
has been the subject of extensive research using methods based on classical ideas from Lyapunov theory
and modern mathematical optimization techniques.
Even for s\witch{} linear systems, arguably the simplest class of hybrid systems,
determining stability is undecidable and approximating the maximal asymptotic growth rate that a trajectory can have is NP-hard \cite{blondel2000boundedness}.
Despite these negative results, the vast range of applications has motivated a wealth of algorithms to approximate
this maximal asymptotic growth rate.

A s\witch{} linear system is characterized by a finite set of matrices $\A \eqdef \{A_1, A_2,\ldots, A_m\} \subset \mathbb{R}^{n \times n}$ and the iteration
\begin{equation}
  \label{eq:switchsys}
  x_k = A_{\sigma_k} x_{k-1}, \quad \sigma_k \in [m]
\end{equation}
where $[m]$ denotes the set $\{1, 2, \ldots, m\}$.

The maximal asymptotic growth rate of this iteration is given by the \emph{joint spectral radius} (JSR).
The JSR $\jsr$ of a finite set of matrices $\A$
is defined as
\[ \jsr = \lim_{k \to \infty} \max_{\sigma \in [m]^k} \|A_{\sigma_k} \cdots A_{\sigma_2} A_{\sigma_1}\|^{1/k}. \]
This definition is independent of the norm used.

The JSR was introduced by Rota and Strang~\cite{rota1960note}
and has many applications such as co-simulation~\cite{gomes2017stable}, wavelets, the capacity of some particular codes, zero-order stability of ordinary differential equations, congestion control in computer networks, curve design and networked and delayed control systems; see \cite{jungers2009joint} for a survey on the JSR and its applications.

In some applications the values that $\sigma_k$ can take in \eqref{eq:switchsys} may depend on $\sigma_{k-1}, \sigma_{k-2},\ldots$.
These constraints are often conveniently represented using a \emph{finite automaton} and the JSR under such constraints
is called \emph{constrained joint spectral radius} (CJSR) \cite{dai2012gel}; an example of constrained s\witch{} system is given by \exemref{run1} and its automaton is illustrated by \figref{run}.
Constrained switched systems are used in a variety of applications including
networked control \cite{brockett2000quantized, zhang2005new} and 
coordination of a network of autonomous agents \cite{jadbabaie2003coordination}. 
Moreover, even if a s\witch{} system is \emph{un}constrained,
studying an associated \emph{constrained} system generated by \emph{path-complete} methods
enhance our ability to analyze the stability~\cite{ahmadi2014joint}
or stabilize~\cite{gomes2018minimally} the original \emph{un}constrained switched system.

The automaton representing the constraints can be represented by a
strongly connected labelled directed graph $\G(\Nodes,\Arcs)$,
possibly with parallel \arcs{}.  The labels are elements of the set
$[m]$ and $\Arcs$ is a subset of $\Nodes \times \Nodes
\times [m]$.  We say that $(u,v,\sigma) \in \Arcs$ if
there is an \arc{} between node $u$ and node $v$ with label $\sigma$.

We use $\Arcs_k$ to denote the subset of $\Arcs^k$ (i.e. the $k$th cartesian power of $\Arcs$) that represents valid paths of length $k$.
The \ktup{k} $(\sigma_1, \sigma_2, \ldots, \sigma_k)$ is said to be $\G$-\emph{admissible} if $\sigma_1, \ldots, \sigma_k$ are the respective labels of a path of length $k$.
We denote the set of all \ktups{k} of $[m]^k$ that are $\G$-admissible as $\G_k$.
The matrix product $A_{\sigma_k}\cdots A_{\sigma_1}$ is written $A_s$ when $s = (\sigma_1, \ldots, \sigma_k)$ or $s$ is a path with these respective labels.

The iteration~\eqref{eq:switchsys} is rewritten as follows to take the automaton into account:
\begin{equation*}
  \label{eq:cswitchsys}
  x_k = A_{\sigma_k} x_{k-1}, \quad (\sigma_1, \ldots, \sigma_k) \in \G_k.
\end{equation*}

The definition of the JSR is generalized as follows for constrained systems.
\begin{mydef}[\cite{dai2012gel}]
  \label{def:cjsr}
  The \emph{constrained joint spectral radius} (CJSR) of
  \defAG{}, denoted as $\cjsr$,
  is
  \begin{equation*}
    \label{eq:jsrtheo}
    \cjsr = \lim_{k \to \infty} \cjsrk,
  \end{equation*}
  where
  \begin{equation}
    \label{eq:cjsrk}
    \cjsrk = \max_{s \in \G_k} \|A_s\|^{1/k}.
  \end{equation}
\end{mydef}

The \emph{arbitrary switching} case~\eqref{eq:switchsys}
can be seen as the particular case when the automaton has only one node and $m$ self-loops with labels $1, \ldots, m$.


\begin{myexem}[Running example]
  \label{exem:run1}
  We borrow the example of \cite[\Sect4]{philippe2016stability}.
  It is based on a state-feedback control that might undergo dropouts in its state feedback.
  The set of matrices $\A$ is composed of the following four matrices
  \begin{align*}
    A_1 & = A+B
    \begin{pmatrix}
      k_1 & k_2
    \end{pmatrix},&
    A_2 & = A+B
    \begin{pmatrix}
      0 & k_2
    \end{pmatrix},\\
    A_3 & = A+B
    \begin{pmatrix}
      k_1 & 0
    \end{pmatrix},&
    A_4 & = A.
  \end{align*}
  where $k_1 = -0.49$, $k_2 = 0.27$,
  \[
    A =
    \begin{pmatrix}
      0.94 & 0.56\\
      0.14 & 0.46
    \end{pmatrix} \text{ and }
    B =
    \begin{pmatrix}
      0\\
      1
    \end{pmatrix}.
  \]

  The corresponding automaton is represented by Figure~\ref{fig:run}.

  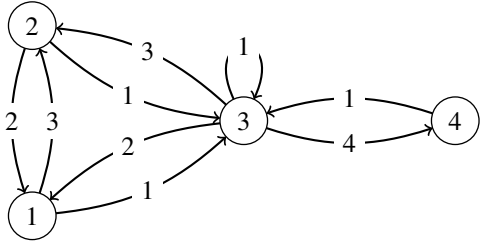
\begin{figure}[!ht]
    \centering
    \begin{tikzpicture}[x=40,y=18]
       \Vertex{1}
       \NO[unit=4](1){2}
       \SOEA[unit=2](2){3}
       \EA[unit=2](3){4}
       \tikzset{EdgeStyle/.style = {->}}
       \tikzset{LoopStyle/.style = {->}}
       \tikzset{EdgeStyle/.append style = {bend right=18}}
       \Edge[label=1](1)(3)
       \Edge[label=3](1)(2)
       \Edge[label=2](2)(1)
       \Edge[label=1](2)(3)
       \Edge[label=2](3)(1)
       \Edge[label=3](3)(2)
       \Edge[label=4](3)(4)
       \draw[thick,->] (3) to [out=115,in=65,looseness=10] node [midway, fill=white] {1} (3);
       \Edge[label=1](4)(3)
     \end{tikzpicture}
     \caption{Automaton for the running example.
     The numbers on the \arcs{} are their respective labels.
     }
     \label{fig:run}
  \end{figure}
\end{myexem}

Approximating the CJSR usually consists in certifying upper bounds $\gamub$ to
the CJSR by exhibiting Lyapunov functions or invariant sets for the matrices
$A_i/\gamub$ (see \secref{entropy} for precise definitions).
The search for such Lyapunov functions can naturally be written as a convex
optimization program using sum of squares (SOS) programming \cite{parrilo2008approximation}.
It turns out that these Lyapunov methods cannot produce an arbitrarily bad CJSR approximation:
bounds are known on the accuracy of the estimate they deliver.
Indeed, the following two bounds have been proved in the \emph{unconstrained} case
for the lowest upper bound $\gamub$ that can be certified using sum of squares polynomials%
\footnote{A polynomial $p(x)$ is a \emph{sum of squares} if there exists some
natural number $k$ and $k$ polynomials $q_i(x)$ such that
$p(x) = q_1^2(x) + \cdots + q_k^2(x)$.} of degree $2d$, denoted $\ujsrsos$:
\begin{align}
  \label{eq:nguarantee}
  \ujsrsos & \leq {n+d-1 \choose d}^{\frac{1}{2d}} \jsr\\
  \label{eq:mguarantee}
  \ujsrsos & \leq m^{\frac{1}{2d}} \jsr.
\end{align}
The two guarantees are incomparable,
as \eqref{eq:nguarantee} depends on the dimension,
and \eqref{eq:mguarantee} depends on the number of matrices.
However, only \eqref{eq:nguarantee} has been generalized in the constrained case yet; see \theoref{matt}.
Our main result is a generalization of the second guarantee: we relate the accuracy
of the SOS-based approximation algorithm with the combinatorial complexity of the automaton.
This complexity is measured by the \emph{entropy} of the language of allowed switching signals.
This new estimate of the accuracy of the SOS technique is always better than the previously existing one for sufficiently large sum of squares degree.
According to the new estimate, the more constrained the system is, the smaller the entropy is and the better the accuracy of the method is.
This shows that, in some sense, it is easier to analyse stability of \emph{constrained} s\witch{} systems than \emph{unconstrained} s\witch{} systems
because the entropy of the language of allowed switching signals is smaller.

Constrained switched systems may also be useful to analyse \emph{abstraction techniques} for complex control systems.
Given a nonlinear system, an \emph{abstraction} of the system can be constructed by
a discretization of the state-space,
such abstraction may enhance our ability to analyse the system \cite{tabuada2009verification}.
The entropy of the language of allowed switching signals of the abstraction is related%
\footnote{The entropy of the abstraction with an $\varepsilon$-discretization measures the growth rate of the number
of cells in which the state could be~\cite[Example~6.3.4]{lind1995introduction}
while the topological entropy is the limsup,
with $\varepsilon \to \infty$, of the growth rate with $n$ of the cardinality of
the largest $(n, \varepsilon)$-\emph{separated} (or the smallest $(n, \varepsilon)$-\emph{spanning}) set;
see \cite{bowen1971entropy} for precise definitions.}
to the \emph{topological entropy} of the nonlinear system~\cite{adler1965topological, bowen1971entropy}.
This suggests that the computational complexity of the abstraction is intrinsically related to
the topological entropy of the nonlinear system and not to the specific choice
of discretization, e.g. the value of $\varepsilon$.
In \cite{yu2010kullback}, the authors use the Kullback-Leibler divergence of the uncertainty induced by a model to measure its fidelity.
They measure the entropy of the \emph{uncertainty} of the noise representing the part of the plant that is not accounted for in the model.
This is similar to our work which
measures the entropy of the \emph{uncertainty} induced by an uncontrolled switching
representing the loss of information due to the discretization.
However, it is fundamentally different as we use this entropy to measure the computational complexity
of the model and not the fidelity of the abstraction.
Indeed, as we have seen, in our work this entropy is related to the topological entropy of the plant and
not to the accuracy of the abstraction.
Other appearances of the entropy in systems and control theory include \cite{byrnes2001generalized, ferrante2008hellinger}; see \cite{pavon2013geometry} for an overview.


In \cite{ahmadi2012joint}, Ahmadi and Parrilo show how to reduce the computation of the JSR
of matrices that are all of rank one to a combinatorial problem,
which coincides with the CJSR of $1 \times 1$ matrices (i.e. scalars).
As a final contribution, we generalize
this approach and give a reduction of the computation of the JSR (or CJSR)
of matrices that are all of rank at most $r$ to the computation of the CJSR of $r \times r$ matrices.

The paper is organized as follows.
In \Sect\ref{sec:entropy}, we give the SOS program searching for Lyapunov functions
and we give our new estimate for its accuracy.
The new bounds explicitly depend on the allowable transitions,
through the graph $G(V,E)$.
In \Sect\ref{sec:lowrank}, we give the low rank reduction mentioned above. 

\paragraph{Reproducibility}
The code used to obtain the results is published on codeocean \cite{c2816dff-2bf4-43b8-bb71-f78ab352d9ae}.
The algorithms are part of the SwitchOnSafety Julia~\cite{bezanson2017julia} package \cite{SwitchOnSafety}
which computes invariant sets for hybrid sytems represented with the HybridSystems package~\cite{HybridSystems}.
The implementation relies on the SumOfSquares~\cite{legat2017sos} and
SetProg~\cite{legat2019set} extensions of JuMP~\cite{dunning2017jump}.
The solver used is Mosek v8~\cite{mosek2017mosek81034}.

\section{Stability and entropy}
\label{sec:entropy}
In this section,
we give the SOS-based method to approximate the CJSR,
we define the entropy of a constrained switching signal and the $p$-radius of a constrained switched system
and we show how the performance guarantee of the method is related to the entropy
of the switching signal and the $p$-radius of the s\witch{} system.

\subsection{Stability}
\label{sec:sos}
As introduced in \cite{parrilo2008approximation} and generalized in \cite{philippe2016stability} for the constrained case,
homogeneous\footnote{A \emph{homogeneous} polynomial \emph{of degree} $2d$ is a polynomial for which the degree of each monomial is $2d$.
The polynomial is called homogeneous as for any real number $\lambda$, we have $p(\lambda x) = \lambda^{2d} p(x)$.}
polynomials of degree $2d$ can be used to certify upper bounds on the CJSR.
\begin{myprop}[{\cite[Theorem~1]{legat2016generating}}]
  \label{prop:pos}
  Consider \defAGe{}.
  Suppose that there exist $|\Nodes|$ strictly positive homogeneous polynomials $p_v(x)$ of degree $2d$
  such that
  \( p_{v}(A_\sigma x) \leq \gamub^{2d}p_{u}(x) \)
  holds for all \arc{} $(u,v,\sigma) \in \Arcs$.
  Then $\cjsr \leq \gamub$.
\end{myprop}

We relax the positivity condition of \propref{pos} by the more tractable sum of squares (SOS) condition and
define $\jsrsos$ as the solution of the following sum of squares program.
\begin{myprog}[Primal]
  \label{prog:primal}
  \begin{align}
    \notag
    \inf_{p_v(x) \in \Rh[x],\gamub \in \R} \gamub\\
    \label{eq:soscons2}
    \gamub^{2d} p_u(x) - p_v(A_\sigma x) & \text{ is SOS}, \quad \forall (u, v, \sigma) \in \Arcs,\\
    \label{eq:soscons1}
    p_v(x) & \text{ is SOS}, \quad \forall v \in \Nodes,\\
    \label{eq:soscons3}
    p_v(x) & \text{ is strictly positive}, \quad \forall v \in \Nodes,\\
    \notag
    \sum_{v \in V} \int_{\mathbb{S}^{n-1}} p_v(x) \dif x & = 1.
  \end{align}
\end{myprog}

\begin{myrem}
  In practice we can replace \eqref{eq:soscons1} and \eqref{eq:soscons3} by ``$p_v(x) - \epsilon\|x\|_2^{2d}$ is SOS''
  for any $\epsilon > 0$.
  This constrains $p_v(x)$ to be in the interior of the SOS cone,
  which is sufficient for $p_v(x)$ to be strictly positive.
  The bounds given in \Sect\ref{sec:quality} are valid
  if $p_v(x)$ is in the interior of the SOS cone.
\end{myrem}

\begin{myrem}
  \label{rem:invscaled}
  The constraint~\eqref{eq:soscons2} is equivalent to
  ``$p_u(x) - p_v(A_\sigma x / \gamub)$ is SOS'' hence the 1-sublevel sets of the
  polynomials $p_v$ provide invariant sets for the matrices $A_\sigma/\gamub$
  as claimed in the introduction.
\end{myrem}

By \propref{pos},
a feasible solution of Program~\ref{prog:primal} gives an upper bound for $\cjsr$,
and thus, for any positive degree $2d$,
\begin{equation}
  \label{eq:upperbound}
  \cjsr \leq \jsrsos.
\end{equation}


\begin{myexem}
  \label{exem:simple1}
  Consider the unconstrained system \cite[Example~2.1]{ahmadi2012joint} with $m=3$:
  \( \A = \{ A_1 = e_1e_2^\Tr , A_2 = e_2e_3^\Tr , A_3 = e_3e_1^\Tr  \} \)
  where $e_i$ denotes the $i$th canonical basis vector.
  For any $d$, a solution to Program~\ref{prog:primal} is given by
  \( (p(x), \gamma) = (x_1^{2d} + x_2^{2d} + x_3^{2d}, 1). \)
\end{myexem}

\begin{myexem}
  \label{exem:run2}
  Let us reconsider our running example; see Example~\ref{exem:run1}.
  The optimal solution of Program~\ref{prog:primal}
  is represented by
  \figref{primal} for $2d = 2$, 4, 10 and 12.
  \begin{figure}[!ht]
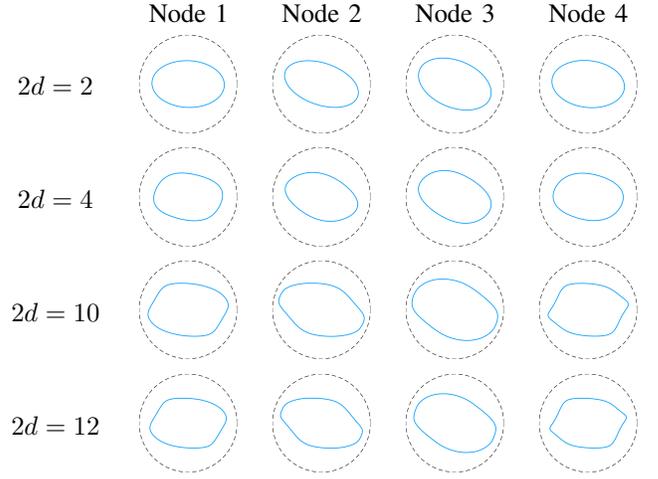

    \centering
    \tabulinesep=0.5mm
    \begin{tabu}{{X[1.0,c,m]X[c,m]X[c,m]X[c,m]X[c,m]}}
      & Node 1 & Node 2 & Node 3 & Node 4\\
      \PrimalFig[0.075]{2}{1}
      \PrimalFig[0.075]{4}{2}
      \PrimalFig[0.075]{10}{5}
      \PrimalFig[0.075]{12}{6}
    \end{tabu}
    \caption[Representation of the solutions to \progref{primal}]{
    Representation of the solutions to \progref{primal}
    with different values of $d$ for the running example.
    The blue curve represents the boundary of the
    1-sublevel set of the optimal
    solution $p_v$ at each node $v \in \Nodes$.
    The dashed curve is the boundary of the unit circle.
    Observe that some sets are not convex.
    Compared to the same experiment done in \cite{legat2016generating},
    the sets represented in this figure correspond to a tigher upper bound
    than the ones reported in \cite[Figure~2]{legat2016generating}.
    This is due to a different implementation of \progref{primal}.
      We use the classical Sum-of-Squares Programming implementation in this paper
      while in \cite[Figure~2]{legat2016generating}, the figure correspond to Common Quadratic Lyapunov Functions (CQLF)
      in a lifted space called the \emph{Veronese embedding}; see \cite[Section~3]{parrilo2008approximation}.
      Using the notation of \cite{parrilo2008approximation},
      the sets of this figure correspond to the upper bound $\rho_{\mathrm{SOS},2d}$ (where SOS stands for Sum-of-Squares)
      and the sets of \cite[Figure~2]{legat2016generating} correspond to the upper bound $\rho_{\mathrm{CQ},2d}$ (where CQ stands for Common Quadratic).
      The inequality $\rho_{\mathrm{SOS},2d} \le \rho_{\mathrm{CQ},2d}$ is proven in \cite[Theorem~5.1]{parrilo2008approximation}.
    }
    \label{fig:primal}
  \end{figure}
\end{myexem}



\subsection{Entropy}
The entropy of a regular language is defined as follows.
\begin{mydef}[{\cite[Definition~4.1.1]{lind1995introduction}}]
    Given a regular language recognized by an automaton $\G$,
    we define the \emph{entropy} of the language as
    \begin{equation}
      \label{eq:hG}
      h(\G) = \lim_{k \to \infty} \frac{1}{k} \log_2|\Paths|.
    \end{equation}
\end{mydef}

The entropy of a language generated by an automaton is easily computable, as we now recall.
The logarithm of the spectral radius of the adjacency matrix of an \emph{irreducible}\footnote{An automaton is \emph{irreducible} if for every pair of nodes $u, v$, there exists a path from $u$ to $v$ accepted by the automaton.} automaton gives the entropy of its \emph{edge shift}.

\begin{mydef}[{\cite[Definition~2.2.5]{lind1995introduction}}]
  \label{def:edgeshift}
  The \emph{edge shift} of an automaton $\G = (\Nodes, \Arcs)$ is the language
  recognized by the automaton $\G' = (\Arcs, \Arcs')$ with the transitions
	$((u,v,\sigma),(v,w,\sigma'),(v,w,\sigma')) \in E'$ for each $(u,v,\sigma), (v,w,\sigma') \in E$.
  We denote the entropy of the edge shift of $\G$ as $h(\Arcs) = h(\G')$.
\end{mydef}

Particularizing equation~\eqref{eq:hG} to the edge shift gives
\begin{equation}
  \label{eq:hE}
  h(\Arcs) = \lim_{k \to \infty} \frac{1}{k} \log_2|\Arcs_k|.
\end{equation}



It turns out that the entropy of the edge shift is equal to the entropy of the language recognized by the automaton if the automaton is \emph{right-resolving}~\cite[Proposition~4.1.13]{lind1995introduction}.

\begin{mydef}[{\cite[Definition~3.3.1]{lind1995introduction}}]
  An automaton $\G$ is \emph{right-resolving} if for every vertex $v$, the outgoing edges have different symbols.
\end{mydef}

Every regular language is recognized by a right-resolving automaton.
Moreover, there are automated ways to obtain such an automaton from a starting representation of a language with an automaton that is not right-resolving \cite[Section~3.3]{lind1995introduction}.

\subsection{Constrained $p$-radius}
The constrained $p$-radius is defined as follows.
\begin{mydef}
  The \emph{constrained $p$-radius} of
  \defAGe{}, denoted as $\cpr$,
  is
  \begin{equation*}
    \cpr =
    \lim_{k \to \infty} \left[|\Arcs_k|^{-1} \sum_{s \in \Arcs_k} \|A_s\|^p\right]^{\frac{1}{pk}}.
  \end{equation*}
  Thus, the CJSR can be defined as the constrained $p$-radius for $p = \infty$.
\end{mydef}

\theoref{pradiusentropy} shows a relation between entropy of the switching signals and the $p$-radius.

\begin{mylem}[{\cite[Corollary~B.5]{legat2017certifying}}]
  \label{lem:summax}
  The limit
  \begin{equation}
    \label{eq:summax}
    \lim_{k \to \infty}
    \left[
      \sum_{s \in \Arcs_k} \|A_s\|^p
    \right]^{\frac{1}{pk}}
  \end{equation}
  converges.
\end{mylem}

\begin{mytheo}
  \label{theo:pradiusentropy}
  Consider \defAG{}.
  The following relation holds
  \[
    \cpr = 2^{-\expfrac{h(\Arcs)}{p}}
    \lim_{k \to \infty}
    \left[
      \sum_{s \in \Arcs_k} \|A_s\|^p
    \right]^{\frac{1}{pk}}.
  \]
  \begin{proof}
    By \lemref{summax}, \eqref{eq:summax} converges and
    by \eqref{eq:hE}, $\lim_{k \to \infty} |\Arcs_k|^{-\frac{1}{pk}} = 2^{-h(\Arcs)/p}$.
%
  \end{proof}
\end{mytheo}

\subsection{Performance guarantees}
\label{sec:quality}
In this section, we provide a new bound that relates
the accuracy of Program~\ref{prog:primal} to
the entropy of the switching signal and
the $p$-radius of the s\witch{} system.

An important property of the $p$-radius is that it is increasing in $p$.
\begin{mylem}[{\cite[Lemma~3.7]{legat2017certifying}}]
  \label{lem:incp}
  Consider \defAG{}.
  For any integers $p \leq q$,
  \begin{multline}
    \label{eq:incp}
    \cpr[p] \leq \cpr[q] \leq \cjsr\\
    \leq 2^{\expfrac{h(\Arcs)}{q}} \cpr[q] \leq 2^{\expfrac{h(\Arcs)}{p}} \cpr[p].
  \end{multline}
\end{mylem}
This Lemma is already known in the unconstrained case where $2^{h(\Arcs)} = m$~\cite{zhou2002p}.

\begin{myrem}
  \label{rem:incp}
  \lemref{incp} shows that the $p$-radius provides an upper and lower bound on the CJSR.
  See \cite{blondel2005computationally, parrilo2008approximation} for methods based on the \emph{veronese liftings}
  computing the $2d$-radius either by computing a spectral radius or by solving a linear program
  (see \cite{ogura2016efficient} for computation algorithms when $p$ is not an even integer).
\end{myrem}

We show the following bound stating that the solution found by \progref{primal} is at least as good as the bound obtained by computing the $2d$-radius (see \lemref{incp}).
\begin{mytheo}
  \label{theo:second}
  Consider \defAG{}. For any positive integer $d$,
  the approximation given by \progref{primal} using homogeneous polynomials of degree $2d$
  satisfies:
  \begin{equation}
    \label{eq:second}
    \jsrsos \leq 2^{\expfrac{h(\Arcs)}{2d}}\cdr \leq 2^{\expfrac{h(\Arcs)}{2d}} \cjsr.
  \end{equation}
\end{mytheo}
Note that the second inequality in \eqref{eq:second} is simply \eqref{eq:incp}.
\theoref{second} is proven at the end of this section.

We can see with \eqref{eq:second} that if $h(\Arcs) = 0$, the approximation is exact.
This corresponds to the case where every node of $\G$ has indegree and outdegree 1.
In that case, the graph forms a cycle of some length $k$ and the CJSR is simply the $k$th root of the spectral radius of the product of the matrices along this cycle.

For the unconstrained switching case, $2^{h(\Arcs)}$ is equal to the number of matrices $m$.
Theorem~\ref{theo:second} is therefore the generalization of \eqref{eq:mguarantee} to the constrained case. 
A generalization of \eqref{eq:nguarantee} to the constrained case was already known
(note that the bound does not take into account the particular structure of the automaton):
\begin{mytheo}[{\cite[Theorem~3.6]{philippe2016stability}}]
  \label{theo:matt}
  Consider \defAGn{} and a positive integer $d$.
  The approximation $\jsrsos$ given by Program~\ref{prog:primal} using homogeneous polynomials of degree $2d$
  satisfies:
  \[ \jsrsos \leq {n+d-1 \choose d}^{\frac{1}{2d}} \cjsr. \]
\end{mytheo}

The results of \theoref{second}, \theoref{matt} and \eqref{eq:upperbound} are summarized by the following corollary.
\begin{mycoro}
  \label{coro:mix}
  Consider \defAGn{} and a positive integer $d$,
  the approximation given by Program~\ref{prog:primal} using homogeneous polynomials of degree $2d$
  satisfies:
  \begin{multline*}
    \max\Big\{{n+d-1 \choose d}^{-\frac{1}{2d}}, 2^{-\expfrac{h(\Arcs)}{2d}}\Big\} \jsrsos\\
    \leq \cjsr \leq \jsrsos.
  \end{multline*}
\end{mycoro}
We see that we can have arbitrary accuracy by increasing $d$.

Our proof technique for \theoref{second} relies on the analysis of an iteration in the vector space of polynomials of degree $2d$.
When this iteration converges, it converges to a feasible solution of Program~\ref{prog:primal}.
By analysing this iteration as affine iterations in this vector space, we derive a sufficient condition for its convergence
and thus an upper bound for $\jsrsos$.

Consider the iteration
\begin{align}
  \notag
  p_{v,0}(x) & = 0,\\
  \label{eq:iter}
  p_{v,k+1}(x) & = q_v(x) + \frac{1}{\iterc} 
  \sum_{(u,v,\sigma) \in \Arcs} p_{u,k}(A_\sigma x), \quad  v \in \Nodes
\end{align}
for fixed homogeneous polynomials $q_v(x)$ of degree $2d$ in $n$ variables (not necessarily different)
and a constant $\iterc > 0$.

When this iteration converges, it converges to a feasible solution of Program~\ref{prog:primal}.
\begin{mylem}
  \label{lem:conv}
  Consider a constant $\iterc > 0$.
  If there exist homogeneous polynomials $q_v(x)$ in the interior of the SOS cone such that
  iteration~\eqref{eq:iter} converges then
  \(
    \jsrsos \leq \iterc^{\frac{1}{2d}}.
  \)
  \begin{proof}
    Suppose the iteration converges to the polynomials $p_{v,\infty}(x)$.
    It is easy to show by induction that $p_{v,k}(x)$ is SOS for all $k$.
    It is trivial for $k = 0$ and if it is true for $k$ then it is also true for $k+1$ by \eqref{eq:iter}.
    Since the SOS cone is closed, $p_{v,\infty}$ is SOS.
    Now by \eqref{eq:iter}, for each $v \in \Nodes$,
    \[ p_{v,\infty}(x) = q_v(x) + \frac{1}{\iterc} 
    \sum_{(u,v,\sigma) \in \Arcs} p_{u,\infty}(A_\sigma x) \]
    so $p_{v,\infty}(x)$ is also in the interior of the SOS cone.
    For each edge $(u,v,\sigma)$, by manipulating the above equation, we have
    \[ \iterc p_{v,\infty}(x) - p_{u,\infty}(A_\sigma x) = \iterc q_v(x) + 
    \sum_{\substack{(u',v,\sigma') \in \Arcs,\\(u', \sigma') \neq (u, \sigma)}} p_{u',\infty}(A_{\sigma'} x) \]
    so $\iterc p_{v,\infty}(x) - p_{u,\infty}(A_\sigma x)$ is SOS.
    Therefore $(\{\, p_{v,\infty}(x) : v \in \Nodes \,\}, \iterc^{\frac{1}{2d}})$ is a feasible solution of Program~\ref{prog:primal}.
  \end{proof}
\end{mylem}

In view of \lemref{conv}, it is thus natural to analyse
under which condition iteration~\eqref{eq:iter} converges.

\begin{proof}[Proof of Theorem~\ref{theo:second}]
  Iteration~\eqref{eq:iter} is an affine map on the vector space of
  homogeneous polynomials of degree $2d$.  It is well known that
  if the convergence is guaranteed when we only retain the linear
  part of the affine map then it is also guaranteed for the affine
  iteration.

  Therefore we can analyse instead the following iteration
  \begin{align*}
    p_{v,0}(x) & = q_v(x),\\
    p_{v,k+1}(x) & = \frac{1}{\iterc} 
    \sum_{(u,v,\sigma) \in \Arcs} p_{u,k}(A_\sigma x), \quad  v \in \Nodes
  \end{align*}

  We can see that
  \begin{equation*}
    p_{v,k}(x) = \frac{1}{\iterc^k}\sum_{s \in \Arcsin_k(v)} q_{s(1)}(A_s x)
  \end{equation*}
  where $s(1)$ denotes the first node of the path $s$.

  Consider a norm $\|\cdot\|$ of $\R^n$ and its corresponding induced matrix norm of $\R^{n \times n}$.
  For each $v \in \Nodes$, we know by continuity of $q_v(x)$ that there exist $\beta_v > 0$ such that
  \( q_v(x) \leq \beta_v\|x\|^{2d} \)
  for all $x \in \R^n$.
  Let $\beta = \max_{v \in \Nodes} \beta_v$,
  then
  \begin{align*}
    p_{v,k}(x)
    & \leq \frac{1}{\iterc^k}\sum_{s \in \Arcsin_k(v)} \beta_{s(1)} \|A_s\|^{2d} \|x\|^{2d}\\
    & \leq \frac{\beta}{\iterc^k} \|x\|^{2d} \sum_{s \in \Arcsin_k(v)} \|A_s\|^{2d}\\
    \sum_{v \in \Nodes} p_{v,k}(x)
    & \leq \frac{\beta}{\iterc^k} \|x\|^{2d} \sum_{s \in \Arcs_k} \|A_s\|^{2d}
  \end{align*}
  By \theoref{pradiusentropy}, if $\iterc > 2^{h(\Arcs)}\cdr^{2d}$, then
  $\lim_{k \to \infty} \sum_{v \in \Nodes} p_{v,k}(x) = 0$
  hence $\lim_{k \to \infty} p_{v,k}(x) = 0$ $\forall v \in \Nodes$ since the polynomials $p_{v,k}$ belong to a proper cone.
  We obtain the result by \lemref{conv}.
\end{proof}

\subsection{Improving the automaton-dependent bounds}

If strong duality holds for a convex problem,
its feasibility is equivalent to the non-existence of an \emph{infeasibility certificate} (see \cite[Section~5.8]{boyd2004convex}).
An infeasibility certificate contains one entry per constraint and if this entry is zero for a given constraint
then the infeasibility certificate remains valid if the constraint is removed from the problem.
In this section, we show how this fact allows to improve the guarantee given by \theoref{second} using the sparsity of the infeasibility certificate.

We show in \cite[Lemma~A.1]{legat2017certifying} that strong duality holds for \progref{primal} with a fixed $\gamub$.
This allows \progref{primal} to be solved by binary search on $\gamub$:
Given a fixed value $\gamma$, the problem is solved with $\gamub = \gamma$;
if a feasible solution is found, it means that $\gamubopt \leq \gamma$,
otherwise, an infeasibility certificate
is found showing that $\gamubopt \geq \gamma$.
By \cororef{mix}, an infeasibility certificate for $\gamma$ provides the following
lower bound certificate on the CJSR:
\[
  \max\Big\{{n+d-1 \choose d}^{-\frac{1}{2d}}, 2^{-\expfrac{h(\Arcs)}{2d}}\Big\} \gamma
  \leq \cjsr.
\]

In \theoref{fifth}
we show a simple way to improve this lower bound certificate
by inspecting the sparsity of the infeasibility certificate.


\begin{mydef}
  Consider \defAGe{}.
  Given an infeasibility certificate $\pmu$ of \progref{primal},
  we denote by \(\Arcs_{\pmu}\) the set of \arcs{} $e \in \Arcs$ such that
  the entry of $\pmu$ corresponding to constraint \eqref{eq:soscons2} with \arc{} $e$ is nonzero.
\end{mydef}

\begin{mytheo}
  \label{theo:fifth}
  Consider \defAGe{}. For any positive integer $d$,
  if there exists an infeasibility certificate $\pmu$ of \progref{primal} with $\gamub = \gamma$
  then
  \begin{equation}
    \label{eq:rmze}
    2^{-\expfrac{h(\Arcs_{\pmu})}{2d}} \gamma \leq \cjsr.
  \end{equation}
  \begin{proof}
    We consider the graph $\G_{\pmu}(\Nodes,\Arcs_{\pmu})$.
    Since the infeasibility certificate $\pmu$ is zero for constraints \eqref{eq:soscons2} with \arcs{} $e \in \Arcs \setminus \Arcs_{\pmu}$,
    $\pmu$ remains a valid infeasibility certificate for \progref{primal} with input $(\G_{\pmu}, \A)$ and $\gamub = \gamma$,
    hence
    \( \gamma \leq \jsrsossub{2d}{\G_{\pmu}}{\A}. \)
    By \theoref{second}, \( 2^{-\expfrac{h(\Arcs_{\pmu})}{2d}} \jsrsossub{2d}{\G_{\pmu}}{\A} \leq \cjsr[\G_{\pmu}] \)
    and since $\Arcs_{\pmu} \subseteq \Arcs$, $\cjsr[\G_{\pmu}] \leq \cjsr$.
    We obtain \eqref{eq:rmze} by combining these three inequalities.
  \end{proof}
\end{mytheo}


\begin{myexem}
  \label{exem:run5}
  Applying the result of this section to the running example gives the result of \figref{run5}.
  The ``Kronecker lift'' lower bound is the bound obtained by using the \emph{Kronecker lift} to transform
  the constrained system with 9 \arcs{} into an unconstrained system with 9 matrices, one per \arc{}.
  The upper bound obtained with both systems is the same \cite[Proposition~3.9]{philippe2016stability}
  hence we can use the guarantee for unconstrained systems \eqref{eq:mguarantee} with $m'=|\Arcs|=9$ for the constrained system.

  The entropy of the switching signal $h(\Arcs)$ used in \theoref{second} is $\log_2(2.61803)$,
  while the value ${n+d-1 \choose d}$ used in \theoref{matt} is $d+1$ since $n=2$.
  Therefore, as we can see on the figure, the lower bound guaranteed by \theoref{matt} is more accurate for $d=1$ only.
  The entropy $h(\Arcs_{\pmu})$ used in \theoref{fifth} is $\log_2(1.61803)$ for $d = 1, 2$ and $\log_2(1.83929)$ for $d = 3, 4, 5, 6$,
  it is more accurate than the three other lower bounds for every $d$.

  The lower bound obtained by computing the $2d$-radius is the most accurate one among all lower bounds for the same $d$ for this example.
  In practice, better lower bounds can be obtained from the solution of \progref{primal} using the techniques of \cite{legat2016generating, legat2017certifying}.

  \begin{figure}[!ht]
    \centering
    \includegraphics[width=0.45\textwidth]{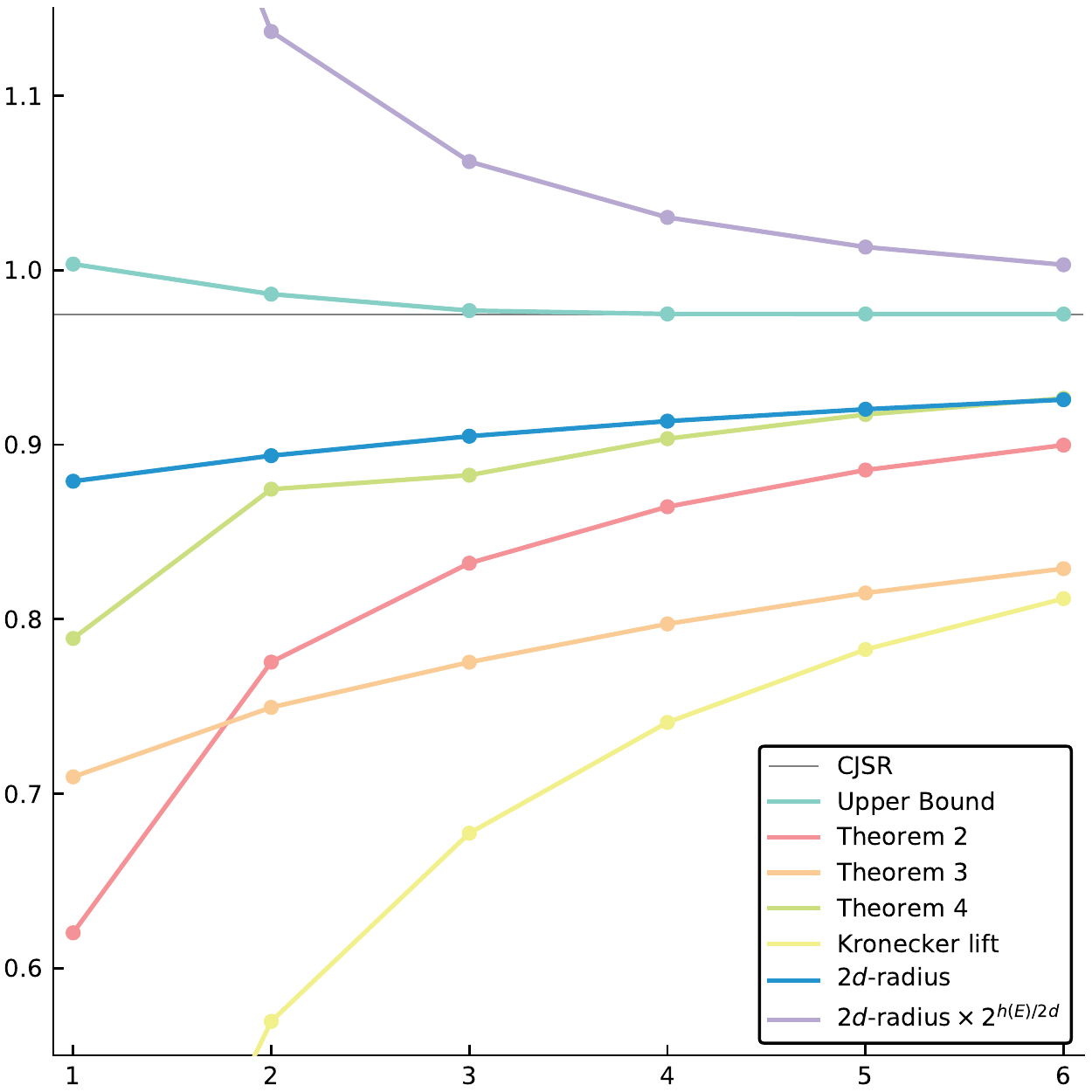}
    \caption{Result of \exemref{run5} for $d = 1, 2, 3, 4, 5, 6$;
      the value of $d$ is given in the horizontal axis.
      The exact value of the CJSR found in \cite{legat2016generating} is represented by the horizontal line.
      The upper and lower bounds given by \lemref{incp} using the $2d$-radius are denoted ``$2d$-radius''.
      The upper bound found by \progref{primal} with polynomials of degree $2d$ is denoted ``Upper Bound''.
      From this upper bound, three lower bounds can be obtained using \theoref{second}, \theoref{matt} and \theoref{fifth}.
      A fourth lower bound can be obtained using the ``Kronecker lift'' as explained in \exemref{run5}.
      }
    \label{fig:run5}
  \end{figure}
\end{myexem}

\section{Low rank reduction}
\label{sec:lowrank}
Suppose we want to compute the CJSR of a finite set of matrices $\A \eqdef \{A_1, \ldots, A_m\} \subset \mathbb{R}^{n \times n}$ of rank at most $r$ constrained by an automaton $\G(V,E)$.
For $\sigma = 1, \ldots, m$, since the matrix $A_\sigma$ has rank at most $r$, there exists $X_\sigma, Y_\sigma \in \mathbb{R}^{n \times r}$ such that $A_\sigma = X_\sigma Y_\sigma^\Tr$.
This can be used to build a new system with matrices of $\R^{r \times r}$ with the same CJSR.
This new system can therefore be used to reduce the computation of the CJSR of a system of low rank matrices to a system of matrices of small size.
Note that in the case $r=1$, it is known that the CJSR is computable in polynomial time~\cite{ahmadi2012joint}.

\begin{mytheo}[Low Rank Reduction]
  \label{theo:lowrank}
  Consider a finite set of matrices $\A \eqdef \{A_1, \ldots, A_m\} \subset \mathbb{R}^{n \times n}$ of rank at most $r$
  constrained by an automaton $\G(V,E)$.

  For a fixed decomposition $A_\sigma = X_\sigma Y_\sigma^T$ for $\sigma = 1, \ldots, m$ where $X_\sigma, Y_\sigma \in \R^{n \times r}$,
  denote the set of matrices $\A' \eqdef \{A_{\sigma_1\sigma_2}' \mid \sigma_1, \sigma_2 = 1, \ldots, m\} \subset \mathbb{R}^{r \times r}$ where $A_{\sigma_1\sigma_2}' = Y_{\sigma_1}^TX_{\sigma_2}$.
  Define the graph $\G'(V',E')$ with $V' \eqdef E$ and
  \[
    E' \eqdef \{\, ((u,v,\sigma_1), (v,w,\sigma_2), \sigma_2\sigma_1) \mid (u,v,\sigma_1),(v,w,\sigma_2) \in E \,\}.
  \]
  Then the two CJSR are the same:
  \( \cjsr = \cjsrp. \)
  \begin{proof}
    As the CJSR does not depend on the norm used, we choose a norm $\|\cdot\|$ that is \emph{submultiplicative},
    that is $\|AB\| \leq \|A\| \|B\|$ for all matrices $A,B$.

    Let $\beta = \max_{\sigma=1}^m\max\{\|X_\sigma\|, \|Y_\sigma^T\|\}$.
    If $\beta = 0$, then $\cjsr = 0 = \cjsrp$.
    Therefore we may assume that $\beta > 0$.
    Consider a positive integer $k$.
    We first show that $[\cjsrk]^k \leq \beta^2[\cjsrkp[k-1]]^{k-1}$
    where $\cjsrk$ is defined in \eqref{eq:cjsrk}.
    For any $\G$-admissible $(\sigma_1, \sigma_2, \ldots, \sigma_k)$, we have
    \begin{equation*}
      \label{eq:redu}
      A_{\sigma_k} \cdots A_{\sigma_2}A_{\sigma_1} = X_{\sigma_k}A_{\sigma_k\sigma_{k-1}}' \cdots A_{\sigma_3\sigma_2}' A_{\sigma_2\sigma_1}'Y_{\sigma_1}^T.
    \end{equation*}
    using
    the submultiplicativity of the norm chosen, we have
    \begin{align*}
      \|A_{\sigma_k} \cdots A_{\sigma_1}\|
      & \leq \|X_{\sigma_k}\| \cdot \|A_{\sigma_k\sigma_{k-1}}' \cdots A_{\sigma_3\sigma_2}' A_{\sigma_2\sigma_1}'\| \cdot \|Y_{\sigma_1}^T\|\\
      & \leq \beta^2\|A_{\sigma_k\sigma_{k-1}}' \cdots A_{\sigma_3\sigma_2}' A_{\sigma_2\sigma_1}'\|\\
      & \leq \beta^2 [\cjsrkp[k-1]]^{k-1}.
    \end{align*}

    The same way, we now show that $[\cjsrkp[k-1]]^{k-1} \leq \beta^2 [\cjsrk[k-2]]^{k-2}$.
    For any $\G'$-admissible $(\sigma_2\sigma_1, \ldots, \sigma_k\sigma_{k-1})$, we have
    \begin{align*}
      \|A_{\sigma_k\sigma_{k-1}}' \cdots A_{\sigma_3\sigma_2}' A_{\sigma_2\sigma_1}'\|
      & \leq \|Y_k^T\| \cdot \|A_{\sigma_{k-1}} \cdots A_{\sigma_2}\| \cdot \|X_1\|\\
      & \leq \beta^2 [\cjsrk[k-2]]^{k-2}.
    \end{align*}

    In summary, we have
    \begin{align*}
      \cjsrk
      & \leq \beta^{\frac{2}{k}} [\cjsrkp[k-1]]^{\frac{k-1}{k}}\\
      & \leq \beta^{\frac{4}{k}} [\cjsrk[k-2]]^{\frac{k-2}{k}}.
    \end{align*}
    Taking the limit $k \to \infty$ we get $\cjsr \leq \cjsrp \leq \cjsr$.
  \end{proof}
\end{mytheo}

\begin{myexem}
  Consider an unconstrained switched system with 2 rank $r$ matrices $A_1, A_2$.
  This system is equivalent to the constrained switched system with automaton represented in \figref{redu1}.
  Its low rank reduction is represented in \figref{redu2}.
  \begin{figure}[!ht]
    \begin{subfigure}{0.17\textwidth}
      \centering
      \begin{tikzpicture}
        \SetGraphUnit{3}
        \GraphInit[vstyle=Dijkstra]
        \SetUpEdge[style={->}]
        \Vertex{0}
        \Loop[labelstyle={above=2pt,fill=white},dist=1.5cm,dir=NO,label={$A_1 = X_1Y_1^T$}](0)
        \Loop[labelstyle={below=2pt,fill=white},dist=1.5cm,dir=SO,label={$A_2 = X_2Y_2^T$}](0)
      \end{tikzpicture}
      \caption{Automaton $\G$.
      We
      have $V = \{0\}$ and $E = \{(0,0,1),(0,0,2)\}$.}
      \label{fig:redu1}
    \end{subfigure}
    \begin{subfigure}{0.31\textwidth}
      \centering
      \begin{tikzpicture}
        \SetGraphUnit{3}
        \SetUpEdge[style={->}]
        \tikzset{EdgeStyle/.style = {distance = 30}}
        \Vertex[NoLabel=false]{1}
        \EA[NoLabel=false](1){2}
        \tikzset{EdgeStyle/.append style = {bend right=60}}
        \Loop[labelstyle={above=0pt},dist=1.5cm,dir=EA,label={$A_{11}' = Y_1^TX_1$}](1)
        \Loop[labelstyle={above=0pt},dist=1.5cm,dir=EA,label={$A_{22}' = Y_2^TX_2$}](2)
        \Edge[labelstyle={above=4pt, fill opacity=0, text opacity=1},label={$A_{21}' = Y_2^\Tr X_1$}](1)(2)
        \Edge[labelstyle={below=4pt, fill opacity=0, text opacity=1},label={$A_{12}' = Y_1^\Tr X_2$}](2)(1)
      \end{tikzpicture}
      \caption{Automaton $\G'$.
      We have $V' \allowbreak= \allowbreak\{1,2\}$ and
      $E' = \{(1,1,11),\allowbreak (1,2,21),\allowbreak (2,1,12),\allowbreak (2,2,22)\}$.}
      \label{fig:redu2}
    \end{subfigure}
    \caption{Simple example of the low rank reduction.}
    \label{fig:redu}
  \end{figure}
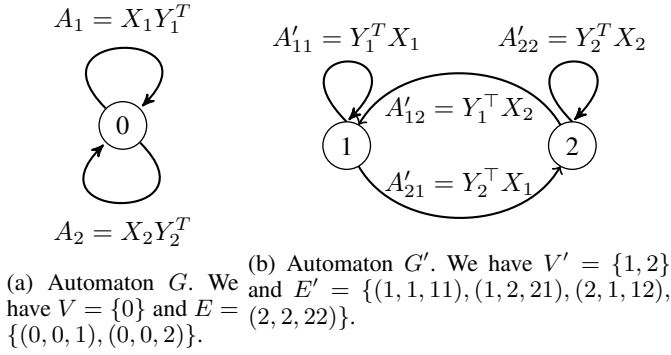
\end{myexem}

\begin{myrem}
  The matrices $X_\sigma, Y_\sigma$ of the factorization $A_\sigma = X_\sigma Y_\sigma^T$ are not unique.
  For any invertible matrix $S \in \mathbb{R}^{r \times r}$, $A_\sigma = (X_\sigma S)(S^{-1}Y_\sigma^T)$ also gives a factorization.
  However, if $\cjsrp$ is approximated using the sum of squares algorithm of \secref{sos},
  any two factorizations will give the same approximation.
  The effect of using $X_\sigma S$ and $Y_\sigma S^{-T}$ instead of $X_\sigma$ and $Y_\sigma$ will simply be a linear change of variable of the polynomial $p_\sigma$;
  see \secref{sos}.

\end{myrem}

What is the impact of this reduction on the computational complexity and accuracy of the approximation ?
The entropy of the language of allowed switching signals is the same for the initial system and the reduced system hence the guarantee in \theoref{second} is the same for both systems.
However, the dimension of the matrices goes from the dimension of the matrices $n$ to their rank $r$
hence for low rank matrices the guarantee in \theoref{matt} is improved.

In terms of computational complexity, there can be up to $m$ \nodes{} and $m^2$ \arcs{} in the automaton of the reduced system.
Therefore, even if the size of the matrices decreases from $n$ to $r$, the number of variables and constraints increases.
This shows that the reduction only decreases the computational complexity if the rank of the matrices is \emph{sufficiently} low.

\section{Conclusion}
This paper uncovers a first relation between the complexity of the discrete dynamic of a hybrid system and
the computational performance of convex optimization methods analysing the stability of its continuous dynamic.
The analysis is performed on discrete linear s\witch{} systems, a subclass of hybrid systems,
but we believe that it should be extended to other classes of hybrid systems such as markovian s\witch{} systems
where the entropy of the discrete dynamics is influenced by transition probabilities.


%


%
%

\ifCLASSOPTIONcaptionsoff
  \newpage
\fi



\bibliographystyle{IEEEtranDOI}
\bibliography{IEEEabrv,2018_TAC-paper}
%


\end{document}